\documentclass[11pt]{article}

\usepackage[top=1.0in, bottom=1.0in, left=1.0in, right=1.0in]{geometry}
\usepackage{amsfonts, amssymb, amsthm, amsmath}
\usepackage{booktabs, threeparttable, multirow}
\usepackage{graphicx}
\usepackage{xcolor}
\usepackage{titlesec} 
\usepackage{caption}[2007/12/23]
\usepackage{afterpage}
\usepackage{algorithm}
\usepackage{algpseudocode}
\usepackage{setspace} 
\usepackage{authblk} 


\usepackage{todonotes}
\usepackage{appendix}
\usepackage{enumerate}
\usepackage{enumitem}
\usepackage{tabularx}
\usepackage{comment}
\usepackage{url}
\usepackage{adjustbox}

\makeatletter
\newcommand{\multiline}[1]{%
  \begin{tabularx}{\dimexpr\linewidth-\ALG@thistlm}[t]{@{}X@{}}
    #1
  \end{tabularx}
}
\makeatother

\newcommand{\algorithmicelsif}{\textbf{else if}}
\algdef{C}[IF]{IF}{NoThenElseIf}[1]{\algorithmicelsif\ #1}

\usepackage{hyperref}

\hypersetup{
    colorlinks=true,
    linkcolor=blue,
    citecolor=blue,
    urlcolor=blue,
    pdfborder={0 0 0}
}

\usepackage[
    backend=biber,
    style=apa,
    uniquename=false
]{biblatex}

\DeclareLanguageMapping{english}{english-apa}
\AtEveryBibitem{\clearfield{url}}
\AtEveryBibitem{\clearfield{doi}}

\newtheorem{definition}{Definition}

\newtheorem{proposition}{Proposition}

\newtheorem{lemma}{Lemma}

\makeatletter
\def\BState{\State\hskip-\ALG@thistlm}
\makeatother

\definecolor{darkblue}{RGB}{0, 51, 102}

\title{
    \textbf{
        \LARGE
        \textcolor{darkblue}{
            Optimization of the Railcar Assignment Problem Using Zone-based Double Deep Reinforcement Learning
        }
    }
}

\author[a]{Ruonan Zhao}
\author[a]{Joseph Geunes\textsuperscript{*}}

\affil[a]{
    Wm Michael Barnes '64 Department of Industrial and Systems Engineering,
    Texas A\&M University
}

\date{}

\titleformat{\section}
  {\color{darkblue}\normalfont\Large\bfseries}
  {\thesection}{1em}{}

\titleformat{\subsection}
  {\color{darkblue}\normalfont\large\bfseries}
  {\thesubsection}{1em}{}

\titleformat{\subsubsection}
  {\color{darkblue}\normalfont\normalsize\bfseries}
  {\thesubsubsection}{1em}{}

\begin{document}
\maketitle

\vspace{10pt}
\begin{abstract}

Railcar switching, or shunting operations decisions play a significant role in the efficient operation of railyard systems, which are in turn critical to the fast and effective movement of goods. In flat yards, switching operations are primarily performed using locomotives to push and pull railcars in order to assemble and disassemble trains. In such settings, railcars with predefined destinations are located across multiple parallel rail tracks, and must be moved, or switched, in order  to form desired outbound trains. This study addresses the Railcar Assignment Problem (RAP) in flat yards with an objective of minimizing the total number of switching movements. We present a novel mixed-integer programming (MIP) model for this problem that incorporates practical operational constraints in rail yards, and demonstrate its NP-hardness. To solve large-scale instances, we propose a comprehensive Zone-based Double Deep Q-Network (Zone-DDQN) heuristic method that integrates railway structure, yard-zone decomposition, and a Double Deep Q-Network (DDQN). The yard-zone decomposition strategy partitions the yard into multiple parallel yard zones, after which the DDQN is applied to solve the problem within each zone individually and sequentially. Computational experiments across small-, medium-, and large-scale yards were conducted on a series of RAP instances. Average results show that the Zone-DDQN heuristic achieves an average optimality gap of \(5.71\%\) across small-scale yard instances. For large-scale yard instances containing more than 150 railcars and 30 tracks, the MIP model was not able to obtain solutions within 24 hours. In contrast, the Zone-DDQN heuristic was able to solve these instances with an average running time of \(214.42\) seconds.\\

\noindent \textbf{Keywords}: Switching Operations; Mixed-Integer Optimization; Deep Reinforcement Learning; Flat Yard Switching

\end{abstract} \hspace{10pt}

\noindent\rule{\textwidth}{0.4pt}

\vspace{5pt}

\noindent\textsuperscript{*}Corresponding author: \texttt{geunes@tamu.edu}.

\newpage
\section{Introduction}
Railway freight transportation is widely recognized as a more environmentally sustainable mode of transport compared to road and air, due to its lower greenhouse gas emissions. As a result, many countries have set strategic goals to expand the proportion of freight moved by rail transportation. The European Commission has set targets for shifting freight from road to rail, aiming to transfer 30\% of road freight traveling more than 300 km by 2030 and exceed 50\% by 2050 through the development of efficient freight corridors \parencite{european2011roadmap}. These trends indicate a global effort to expand railway freight systems. However, a major operational bottleneck in railway freight systems occurs in yards, where railcars are sorted, disassembled, and assembled to create trains. Freight railcars in North America are likely to have dwell times exceeding 20 to 30 hours in switching yards, and up to 37 hours during congested periods \parencite{BTS_2025_TSAR}. Better coordination of switching movements in yards is a key factor in reducing these dwell times and improving throughput and operational efficiency.  

Switching, or marshaling operations are critical in determining overall yard efficiency. These operations involve rearranging railcar(s) between tracks to form outbound trains. Rail yards are typically categorized into three types: flat yards, gravity yards, and hump yards. In flat yards, all switching operations are carried out solely by locomotives, unlike gravity yards and hump yards, where gravity assists in car movement, making the switching process more efficient. Within each yard, the track configuration can be either stub or through. In stub tracks, all switching movements occur from a single end, resulting in a stack structure with a last-in-first-out (LIFO) property. In contrast, a so-called through track permits entry and exit at either end, and its queue structure allows first-in-first-out (FIFO) access, providing increased operational flexibility.

The switching process can involve different objectives, such as forming outbound trains with a specified sequence \parencite{schduleinthemorning1999,trainmarshalling2000} or extracting specific railcars for maintenance \parencite{fyRetrieving2018}. These types of operations have been widely studied in the literature, with commonly used approaches including mixed-integer programming (MIP) models \parencite{LK2005}, dynamic programming (DP) \parencite{falsafain2019novel}, approximate dynamic programming (ADP) \parencite{switchingproject}, or heuristic approaches \parencite{haahr2016matheuristic}. In recent years, reinforcement learning (RL)–based methodologies have also been explored for such problems, where policies are learned through interaction with the environment for sequential decision-making tasks. For instance, \textcite{hiroshima2012reinforcement} proposed an RL method to minimize the total processing time to form one outbound train with the desired order. However, RL methods face scalability challenges because the state and action spaces increase rapidly with the number of tracks and railcars. In contrast, railway domain heuristics can exploit structural properties of the yard to produce solutions quickly, although they may be short-sighted.

In this paper, we study the RAP in flat yards with stub tracks, where the objective is to determine a sequence of switching moves that consolidates railcars with the same destination onto a single track while minimizing the total number of switching movements. To address this problem, we first develop an MIP model that considers the detailed operational constraints in the yards. Second, we construct a Zone-based Double Deep Q-Network (Zone-DDQN) algorithm that integrates railway-domain heuristics, zone decomposition, and deep reinforcement learning. 

Both methods presented in this paper can be extended to other similar sorting problems sharing the stack structure. For example, in water or ball sorting problems, units of different colors are placed in containers that require LIFO access, and the objective is to move units so that each container holds only a single color \parencite{ito2023sorting}. Another practical application is the container relocation problem in terminal yards. Containers are stacked in vertical columns, and only the top container in a stack can be accessed directly, leading to a LIFO structure \parencite{shin2025DQL_Container,tang2024DQL_Container}.

The contributions of this paper are listed as follows:
\begin{enumerate}[label=(\roman*)]
    \item We formulate the RAP in stub flat yards as a sequential railcar switching optimization problem and develop an MIP model with an objective function of minimizing the total number of switching movements.

    \item The NP-hardness for the RAP is established by showing that a restricted version of RAP is equivalent to the water sorting problem, which has been proven to be NP-complete \parencite{ito2023sorting}.

    \item We propose a comprehensive RL framework based on Double Deep Q-Networks (DDQN) for the RAP, and develop a Zone-DDQN algorithm that integrates railway domain merging heuristics, yard-zone decomposition, and DDQN.
  
    \item Computational experiments on 30 generated RAP instances demonstrate that the proposed Zone-DDQN framework can efficiently generate solutions with an average optimality gap of 5.71\%.
\end{enumerate}

This paper is organized as follows. The next section summarizes related literature. In Section~\ref{sec:Problemdef}, we present a formal definitions of the RAP. 
Section~\ref{sec:mip} presents our MIP model for the RAP, where detailed operational constraints are considered. Then, Section~\ref{sec:Zone-DDQN_alg} first establishes the NP-hardness of the RAP and subsequently presents the comprehensive Zone-DDQN framework. Section~\ref{sec:c_results} presents the numerical results. Lastly, Section~\ref{sec:conclusion} concludes the paper and suggests future directions for research on shunting operations.

\section{Literature Review}
\label{sec:literature}
Switching operations have been widely studied in the literature due to their direct impact on yard efficiency, dwell time, and resource allocation and utilization. This problem class involves complex combinatorial structures and sequential decision-making under operational constraints such as track capacity, movement feasibility, and locomotive movements. Different types of railway yards use different switching mechanisms. In hump or gravity yards, railcars are moved primarily by gravity, whereas flat yards rely on locomotives to perform switching operations. \textcite{boysen_shuntingreview} provide a comprehensive review of shunting yard operations, including single-stage and multi-stage sorting strategies. In multi-stage sorting, railcars may be reshunted, whereas in single-stage sorting, railcars can only be shunted once. 

A major part of the literature formulates related switching problems using classical operations research models. \textcite{zhang2018integer_programming} propose an integer programming model to assign railcars to classification tracks, with the objective of minimizing reswitching operations caused by undesired railcar sequences. \textcite{fyRetrieving2018} investigate a railcar retrieval problem for maintenance operations, where required numbers of railcars with distinct types need to be retrieved from the flat yard. They formulate this problem using an MIP model that minimizes retrieval cost, where the cost associated with each railcar depends on its position: railcars located at the switching end incur a lower retrieval cost, while those positioned further away incur a higher cost.
\textcite{switchingproject} consider the railcar shunting problem with a goal of relocating all railcars to their designated departure tracks. 
They formulate an MIP model that minimizes total shunting cost in a single-locomotive setting, while incorporating operational constraints such as the LIFO access of railcars in stub yards. \textcite{han_optimizing} propose a graph model, wherein nodes represent railcars and arcs represent actions. They then construct a binary integer programming model for the graph with a goal of maximizing the cumulative reward associated with the selected actions.

Another line of research focuses on the use of heuristic algorithms and simulation methods. \textcite{Kozachenko02012021} study the multiple-group train formation problem, where railcars need to be rearranged into a specified order using a single locomotive. The objective seeks to minimize the overall shunting time required. They formulate the problem on a directed graph and construct a shortest path algorithm to address it. \textcite{ZIEN2025heu} propose a rolling horizon heuristic to solve the railcar sorting problem in yards while considering restricted track availability and finite track capacities.
\textcite{Gaiayardsimu} provide a simulation model for evaluating flat yard operations. The yard is partitioned into multiple segments—such as arrival, classification, and departure tracks—to study both their individual and combined performance. 

In recent years, RL, particularly Q-learning has been applied to switching problems to learn decision policies through interaction with the environment. RL models learn by interacting with the environment, allowing them to adapt their decisions and improve performance over time \parencite{sutton1998reinforcement}. \textcite{hiroshima2012reinforcement} studies the problem of rearranging railcars, initially distributed across multiple tracks, into a specified order on a single main track to form one outbound train. They construct a Q-learning approach, with an objective that minimizes  total processing time. \textcite{peer2018shunting} represent the yard state as an image, and solve the Train Unit Shunting Problem using deep reinforcement learning. However, in their problem formulation, temporarily relocating a train from one track to another is not permitted. \textcite{2ndproject} propose a hybrid heuristic-RL method to minimize the total shunting cost associated with locomotive moves, allowing the movement of individual railcars or consecutive railcars between tracks on both stub and through tracks.
 
The study most closely related to our work is the work of \textcite{hirashima2011reinforcementloco}, who proposes an RL method to minimize the total distance traveled by a locomotive while moving all railcars, initially distributed across multiple subtracks, into one main track to form one outbound train. They formulate the problem as a Markov Decision Process (MDP) and apply Q-learning to obtain switching policies. In their setting, railcar movements are restricted to transfers from a subtrack to the main track or from one subtrack to another subtrack. Additionally, their computational experiments are limited to relatively small yard instances consisting of only 12 tracks, each with a capacity of 6 railcars, and a total of 36 railcars in the yard. Traditional Q-learning methods suffer from scalability limitations when the state-action space becomes large \parencite{mnih2015human}, since they require maintaining and updating Q-values for each state-action pair. The issue is particularly challenging in large railway yards with numerous tracks, railcars, and feasible switching actions. To address this challenge, we propose a novel Zone-DDQN framework to integrate railway-domain heuristics, zone decomposition, and double deep reinforcement learning to solve the RAP. The strength of the Zone-DDQN framework lies in its flexibility and adaptability: (i) it allows flexible railcar movements between any pair of tracks; (ii) it does not impose a fixed number of outbound trains, but instead dynamically forms one or multiple outbound trains according to yard configurations; and (iii) its stack-based structure can also be extended to related applications such as the container relocation problem \parencite{jiang2021new} and water sorting puzzle \parencite{ito2023sorting}.

\section{RAP Problem Definition}
\label{sec:Problemdef}
This section presents a formal description of the RAP. We begin by introducing the terminology used throughout the problem. A stub yard contains multiple parallel tracks that are connected via ladder tracks on the one side, enabling locomotives to transfer railcars across tracks (see Figure \ref{rapyard} for an example layout). Each track can store a sequence of railcars subject to its capacity limit, and railcars can only enter or leave a track from the switch end connected to the ladder track. The opposite side of a track is referred to as the dead end, where railcars cannot enter or leave. A source track is defined as the track from which railcar(s) are moved, while a receiver track is defined as the track to which railcar(s) are moved. One switching move is defined as moving an individual railcar or a set of consecutive railcars from a source track to a receiver track. The objective of the RAP is to determine a sequence of feasible switching moves that rearranges railcars into outbound trains, where each outbound train consists of railcars sharing the same destination, while minimizing the total switching movements.

The example yard layout shown in Figure~\ref{rapyard}(a) consists of six tracks and two ladder tracks. Railcars are represented as colored rectangles, where each color indicates a distinct destination. A locomotive is represented as a black rectangle. Initially, railcars of different destinations are distributed randomly throughout the yard. Figure~\ref{rapyard}$(b)$ shows one feasible final configuration corresponding to the initial yard layout presented in Figure~\ref{rapyard}$(a)$, where each track contains railcars sharing the same destination. Note that the railcar sequence within each track in the final layout is not required to follow any specific order; that is, any ordering of railcars with the same destination (color) within a track is acceptable.

\begin{figure}[h]
\centering
\includegraphics[width=0.88\textwidth]{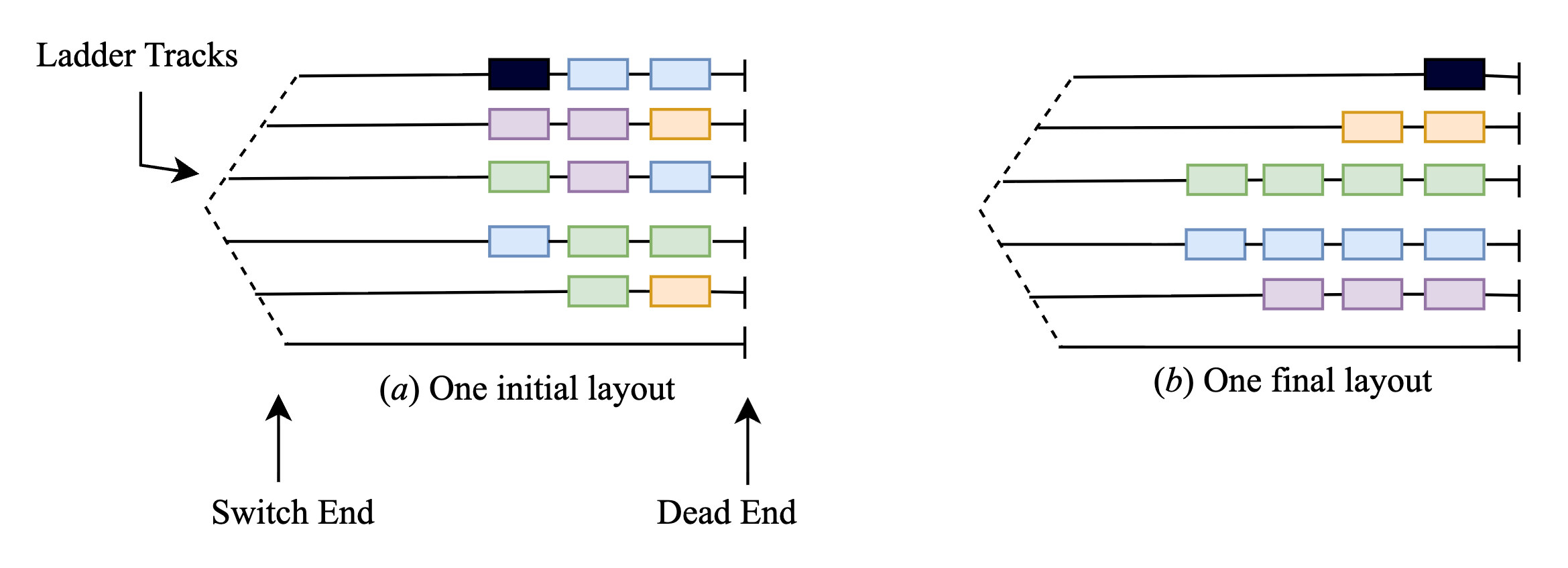}
\caption{Railyard configuration example}
\label{rapyard}
\end{figure}

According to the context described above, the RAP is defined as follows.
\begin{definition}[\textbf{Railcar Assignment Problem (RAP)}]
We are given:
(i) a set of tracks $K$, each with common capacity $H\in\mathbb{Z}_{\ge 1}$,
(ii) a set of railcars $R$,
(iii) a set of destinations $D$,
(iv) an initial layout specifying, for each track, the ordered list of cars from dead end to switch end, where each car $r \in R$ has a destination $d(r) \in D$.

We seek a sequence of switching moves such that all cars on an track have the same color. In each move, we choose a source track and a receiver track, remove a contiguous set of cars located in the source track via the switch end (i.e., the leftmost \(m\) cars for some \(m\ge 1\)), and place these cars into the receiver track, preserving the ordering of the moved cars. The receiver track must have sufficient remaining capacity to accommodate each move. The objective is to minimize the total number of switching moves required to reach a terminal state in which all railcars sharing the same destination are occupy a single track.
\end{definition}

\section{RAP MIP Model}
\label{sec:mip}
As previously discussed, the RAP aims to minimize the total number of switching moves required to rearrange railcars with the same destination into one track. To achieve this objective, this section presents an MIP model for the RAP. The model assumptions are as follows: (i) there is one locomotive in the system; (ii) at most one switching move may occur per time period; (iii) railcars on each track occupy consecutive slots starting from the dead end, with no empty slots between adjacent railcars; (iv) the total number of railcars for each destination is less than or equal to the track capacity; (v) the total number of destination labels is less than or equal to the number of tracks. Slots on each track are indexed from 1 to the track capacity, starting from the dead end and increasing toward the switch end. The slot information for a track with capacity 6 containing 3 railcars is illustrated in Figure~\ref{slot}. Note that the white blocks in slots 4, 5, 6 represent empty slots.

\begin{figure}[h]
\centering
\includegraphics[width=0.45\textwidth]{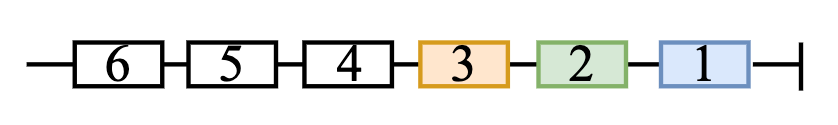}
\caption{Track slot example}
\label{slot}
\end{figure}

Our modeling approach divides time into a set $T$ of discrete time periods, where each switching move requires one time period. The required sets and input parameters are described in Table \ref{tab:notation}. The MIP decision variables for the RAP are classified into several categories based on their functions within the model. The model’s decision variables are provided in Table \ref{tab:RAP_MIP_variables}.

\begin{table}[htbp]
\centering
\caption{Definition for sets and parameters.}
\label{tab:notation}
\begin{tabular}{p{3cm} p{11cm}}
\toprule
\textbf{Notation} & \textbf{Definition} \\
\midrule

\multicolumn{2}{l}{\textbf{Sets}} \\

$K$ 
& Set of tracks: $i,j \in K$. \\

$L$ 
& Set of slots: $l \in L$. \\

$D$ 
& Set of railcar destinations: $d \in D$. \\

$T$ 
& Set of the decision time periods: $t \in T$, where $T=\{0,1,\dots,T_{\max}\}$ \\
\midrule
\multicolumn{2}{l}{\textbf{Input parameters}} \\

$H$
& Track capacity \\
\bottomrule
\end{tabular}
\end{table}

\begin{table}[htbp]
\centering
\caption{MIP decision variables.}
\label{tab:RAP_MIP_variables}
\begin{tabular}{p{4.5cm} p{2cm} p{8.5cm}}
\toprule
\textbf{Decision Variable} & \textbf{Type} & \textbf{Description} \\
\midrule

\multicolumn{3}{l}{\textbf{State variables}} \\

$y_{kl t}$ 
& Binary
& if slot $l$ on track $k$ is occupied at the end of period $t$, then $y_{kl t}=1$, otherwise $y_{kl t}=0$. \\

$v_{kl dt}$ 
& Binary
& if slot $l$ on track $k$ holds a railcar with destination $d$ at the end of period $t$, then $v_{kl dt}=1$, otherwise $v_{kl dt}=0$. \\

$h_{kt}\in\{0,1,\dots,H\}$ 
& Integer
& Number of railcars on track $k$ at the end of period $t$. \\

\midrule
\multicolumn{3}{l}{\textbf{Move variables}} \\

$x_{ijt}$ 
& Binary
& if the locomotive executes one switching move from track $i$ to track $j$ in period $t$, then $x_{ijt}=1$, otherwise $x_{ijt}=0$. \\

$a_{ijl t}$ 
& Binary
& if railcar at slot $l$ is moved from track $i$ to track $j$ in period $t$, then $a_{ijlt}=1$, otherwise $a_{ijl t}=0$.  \\

$b_{ijl t}$ 
& Binary
& if slot $l$ of track $j$ becomes occupied by a move from track $i$ in period $t$, then $b_{ijl t}=1$, otherwise $b_{ijl t}=0$.  \\

$m_{ijt}\in\{0,1,\dots,H\}$ 
& Integer
& Number of railcars moved from track $i$ to track $j$ in period $t$. \\

$\lambda_{ijqt}$ 
& Binary
& Index shift $q\in\Delta$ selected for move $(i\rightarrow j)$ in period $t$. \\

\midrule
\multicolumn{3}{l}{\textbf{Destination assignment variables}} \\

$s_{ijl dt}$ 
& Binary
& if railcar with destination $d$ at position $l$ on track $i$ is transferred to track $j$ in period $t$, then $s_{ijl dt}=1$, otherwise $s_{ijl dt}=0$. \\

$r_{ijl dt}$ 
& Binary
& if railcar with destination $d$ is placed at slot $l$ on track $j$ in period $t$ after being transferred from track $i$, then $r_{ijl dt}=1$, otherwise $r_{ijl dt}=0$. \\

\midrule
\multicolumn{3}{l}{\textbf{Completion-time variables}} \\

$w_t$ 
& Binary
& if the system reaches a valid terminal state by time period $t$, then $w_t=1$, otherwise $w_t=0$. \\

$u_t$ 
& Binary
& if period $t$ is the earliest period a valid terminal state is reached, then $u_t=1$, otherwise $u_t=0$. \\

\midrule
\multicolumn{3}{l}{\textbf{Terminal destination selectors}} \\

$\zeta_{kd}$ 
& Binary
& if destination $d$ is selected for track $k$ in the terminal state, then $\zeta_{kd}=1$, otherwise $\zeta_{kd}=0$. \\

\bottomrule
\end{tabular}
\end{table}

Based on the above definitions, the proposed MIP model for RSP is as follows.\\

1) Objective function
\begin{equation}
\min \ \sum_{t\in T}\ \sum_{\substack{i,j\in K\\ i\neq j}} x_{ijt}
\label{eq:objfun}
\end{equation}

The objective function minimizes the overall number of switching moves.

\noindent 2) State definition constraints
\begin{align}
&\sum_{d\in D} v_{kl dt} = y_{kl t} && \forall k\in K,l\in L,t\in T, \label{eq:state1}\\
&y_{k,l-1,t}\ \ge\ y_{kl t} && \forall k\in K,\ l=2,\dots,H,\ \forall t\in T, \label{eq:nogap}\\
&h_{kt}=\sum_{l=1}^{H} y_{kl t} && \forall k\in K,t\in T. \label{eq:height}
\end{align}

Constraint~\eqref{eq:state1} ensures that each position on a track is either empty or occupied by exactly one railcar destination type in every time period. Constraint~\eqref{eq:nogap} ensures that railcars on each track occupy consecutive positions starting from the dead end. In particular, if position $l$ is occupied at time $t$, then position $l-1$ must also be occupied at time $t$, thereby enforcing a contiguous stack structure. Constraint~\eqref{eq:height} defines the track ``height'' variable $h_{kt}$ as the total number of occupied positions on track $k$ at time $t$.\\ 

\noindent 3) Move selection constraints
\begin{align}
&\sum_{\substack{i,j\in K\\ i\neq j}} x_{ijt} \le 1-w_t && \forall t \in T \setminus \{T_{\max}\}, \label{eq:onemove}\\
&x_{ijt} \le y_{i1t} && \forall i,j\in K, i\neq j,\ t \in T \setminus \{T_{\max}\}, \label{eq:srcnonempty}\\
&a_{ijl t} \le x_{ijt} && \forall i,j\in K, i\neq j,\ l\in L,\ t \in T \setminus \{T_{\max}\}, \label{eq:activatea}\\
& b_{ijl t} \le x_{ijt} && \forall i,j\in K, i\neq j,\ l\in L,\ t \in T \setminus \{T_{\max}\}, \label{eq:activateb}\\
&m_{ijt} = \sum_{l=1}^{H} a_{ijl t} && \forall i,j\in K, i\neq j,\ t \in T \setminus \{T_{\max}\}. \label{eq:mdef}\\
& m_{ijt}\ge x_{ijt}
&& \forall i,j\in K, i\neq j,\ t \in T \setminus \{T_{\max}\}. \label{eq:m_x_relation}
\end{align}

Constraint~\eqref{eq:onemove} enforces that at most one switching move can be performed during each time period before the yard enters a terminal state. Specifically, if the yard system reaches terminal state at time $t$ (i.e., $w_t=1$), then the right-hand side becomes $0$ and no move is allowed in period $t$. If $w_t=0$, the constraint reduces to $\sum_{i\ne j} x_{ijt}\le 1$, allowing at most one move. Constraint~\eqref{eq:srcnonempty} prevents moves from an empty source track, i.e., a move from track $i$ can occur only if track $i$ is nonempty at time $t$ (slot $1$ is occupied by a car). Constraints~\eqref{eq:activatea} and ~\eqref{eq:activateb} ensure that a railcar cannot move from track $i$ to $j$ unless the locomotive moves from $i$ to $j$ in period $t$. Constraint~\eqref{eq:mdef} sets the value of $m_{ijt}$ by summing $a_{ijl t}$ over all slots $l=1,\dots,H$. Thus, $m_{ijt}$ indicates the overall number of railcars moved from track $i$ to $j$ in time period $t$. Constraint~\eqref{eq:m_x_relation} ensures that whenever a move is selected, at least one car must be moved.\\

\noindent 4) Source track selection constraints
\begin{align}
&a_{ijl t} \le y_{il t} && \forall i,j\in K,i\neq j,\ l\in L,\ t \in T \setminus \{T_{\max}\}, \label{eq:asrcfilled}\\
&a_{ij,l,t} \le a_{ij,l+1,t} + (1 - y_{i,l+1,t})
&& \forall i,j\in K,i\neq j,\ l=1,\dots,H-1,\ t \in T \setminus \{T_{\max}\}. \label{eq:acontig}
\end{align}

Constraint~\eqref{eq:asrcfilled} ensures that position  \(l\) on track $i$ can be selected for movement only if that position is currently occupied by a railcar. Constraint~\eqref{eq:acontig} enforces that the selected railcars form a consecutive group at the switch end of the source track, which is due to the stub track structure. Specifically, a railcar in slot \(l\) on track \(i\) cannot be selected for movement if the railcar in slot \(l+1\) on track \(i\) is occupied but not selected for movement.

\noindent 5) Receiver track selection constraints
\begin{align}
&b_{ijl t} \le 1 - y_{jlt} && \forall i,j\in K,i\neq j,\ l\in L,\ t \in T \setminus \{T_{\max}\}, \label{eq:emptydest}\\
&b_{ijl t} \le b_{ij,l-1,t} + y_{j,l-1,t}
&& \forall i,j\in K,i\neq j,\ l=2,\dots,H,\ t \in T \setminus \{T_{\max}\}, \label{eq:Contiguousfill}\\
&\sum_{l=1}^{H} b_{ijlt} = m_{ijt}
&& \forall i,j\in K,i\neq j,\ t \in T \setminus \{T_{\max}\}. \label{eq:numbermatch}
\end{align}

Constraint~\eqref{eq:emptydest} requires that railcars can only be placed into the currently empty positions of receiver track $j$. Recall that railcars on each track are assumed to occupy consecutive slots beginning from the dead end. Therefore, a slot $l$ on track $j$ may be filled only if slot $l-1$ is either already occupied or is simultaneously filled during the same move. Constraint~\eqref{eq:Contiguousfill} ensures that railcars added to receiving track $j$ occupy consecutive slots without gaps. Constraint~\eqref{eq:numbermatch} matches the number of filled slots on the receiving track to the number of railcars moved.\\ 

\noindent 6) State transition constraints
\begin{align}
&y_{kl,\,t+1} \ =\ y_{kl t}
\ + \sum_{\substack{i\in K\\ i\neq k}} b_{iklt}
\ - \sum_{\substack{j\in K\\ j\neq k}} a_{kjlt}
&& \forall k \in K,l\in L,\ t \in T \setminus \{T_{\max}\}. \label{eq:ytrans}
\end{align}
\begin{align}
&v_{kl d,\,t+1} \ =\ v_{kl dt}
\ - \sum_{\substack{j\in K\\ j\neq k}} s_{kjl dt}
\ + \sum_{\substack{i\in K\\ i\neq k}} r_{ikl dt}
&& \forall k\in K,l\in L,d\in D,\ t \in T \setminus \{T_{\max}\}. \label{eq:vtrans}
\end{align}

From time $t$ to $t+1$, constraint~\eqref{eq:ytrans} updates whether slot $l$ on track $k$ is occupied. Constraint~\eqref{eq:vtrans} updates the destination associated with slot $l$ on track $k$ from time $t$ to $t+1$. The term $v_{kl dt}$ represents the current destination assignment at time $t$. The term $\sum_{j\neq k} s_{kjl dt}$ removes the assignment if the railcar occupying slot $l$ on track $k$ is transferred from track $k$ to another track during period $t$. The term $\sum_{i\neq k} r_{ikl dt}$ assigns a destination to slot $l$ on track $k$ when a railcar transferred from another track is placed into that slot during period $t$. Therefore, this constraint maintains the correct destination assignment after each switching move.\\

\noindent 7) Destination assignment constraints
\begin{align}
&s_{ijl dt} \le a_{ijl t}
&& \forall i,j\in K,i\neq j,\ l\in L,d\in D,\ t \in T \setminus \{T_{\max}\}, \label{eq:slink1_a}\\
&s_{ijl dt} \le v_{il dt}
&& \forall i,j\in K, i\neq j,\ l\in L,d\in D,\ t \in T \setminus \{T_{\max}\}, \label{eq:slink1_v}\\
&\sum_{d\in D} s_{ijl dt} = a_{ijl t}
&& \forall i,j\in K, i\neq j,\ l\in L,\ t \in T \setminus \{T_{\max}\}, \label{eq:slink2}\\
&r_{ijl dt} \le b_{ijl t}
&& \forall i,j\in K,i\neq j,\ l\in L,d\in D,\ t \in T \setminus \{T_{\max}\}, \label{eq:rlink1}\\
&\sum_{d\in D} r_{ijl dt} = b_{ijl t}
&& \forall i,j\in K,i\neq j,\ l\in L,\ t \in T \setminus \{T_{\max}\}. \label{eq:rlink2}
\end{align}

Constraints~\eqref{eq:slink1_a}-\eqref{eq:slink2} define the destination associated with each railcar removed from track $i$ and transferred to track $j$. Since multiple railcars with different destinations may be moved simultaneously, the formulation tracks the destination assignment at each slot individually. If slot $l$ on track $i$ is selected to be moved from track $i$ to $j$ at time $t$ (i.e., $a_{ijlt}=1$), then exactly one destination $d$ must be associated with the move (by requiring $s_{ijl dt}=1$ for exaclty one destination). Constraints~\eqref{eq:slink1_a} and \eqref{eq:slink1_v} ensure that a railcar with destination $d$ at slot $l$ on track $i$ can be transferred to track $j$ in period $t$ only if there is a move from track $i$ to $j$ at time $t$, and the car in slot has destination $d$. Constraint~\eqref{eq:slink2} ensures that each moved railcar is associated with exactly one destination. Constraint~\eqref{eq:rlink1} ensures that a destination assignment variable $r_{ijl dt}$ cannot be equal to $1$ unless slot $l$ on track $j$ is selected to receive a railcar from track $i$ in time $t$. Constraint~\eqref{eq:rlink2} ensures that each slot $l$ on track $j$ receiving a railcar is associated with exactly one destination.\\

\noindent 8) Slot update constraints

The slot index change for any railcar that is moved in period $t$ is characterized by Lemma \ref{lem:slotChange}. 

\begin{lemma}\label{lem:slotChange}
Given $h_{j,t}$ and $h_{it}$ railcars on tracks $i$ and $j$ at the end of period $t$, respectively, if $m_{ijt}$ railcars move from track $i$ to track $j$ in period $t$, then the resulting slot index changes for each railcar moved are given by:
\begin{eqnarray}
q &= h_{jt} - h_{it} + m_{ijt},& 
\forall\, j,i \in K,j\neq i, t \in T. \label{eq:30}
\end{eqnarray}
The value of $q$ belongs to the set $\Delta=\{-(H-1),\dots,H-1\}$.
\end{lemma}
\begin{proof}
Please see Appendix \ref{sec:proofs_appen}.
\end{proof}

We use the binary variable $\lambda_{ijqt}$ to ensure a correct change in the slot index number associated with each railcar that is moved during any period using the following set of constraints.
\begin{eqnarray}
&\sum_{q\in \Delta} \lambda_{ijqt} = x_{ijt}
& \forall i,j\in K,\ i\neq j,\ t \in T \setminus \{T_{\max}\}, \label{eq:lambdachoose}\\
&-M(1-x_{ijt})
\le
h_{jt}-h_{it}+m_{ijt}
-\sum_{q\in\Delta} q\,\lambda_{ijqt}
& \forall i,j\in K,\ i\neq j,\ t \in T \setminus \{T_{\max}\}, \label{eq:qdef}\\
&h_{jt}-h_{it}+m_{ijt}
-\sum_{q\in\Delta} q\,\lambda_{ijqt}
\le
M(1-x_{ijt})
& \forall i,j\in K,\ i\neq j,\ t \in T \setminus \{T_{\max}\}, \label{eq:qdef2}\\
&r_{ij,l+q,dt}
\ge
s_{ijl dt} + \lambda_{ijqt} - 1
& \forall i,j\in K,\ i\neq j,\ q\in\Delta,\nonumber\\
&& \forall l:\ l+q\in L,\ \forall d \in D,\ t \in T \setminus \{T_{\max}\}, \label{eq:mapLB}\\
&r_{ij,l+q,dt}
\le
s_{ijl dt} + (1-\lambda_{ijqt})
& \forall i,j\in K,\ i\neq j,\ q\in\Delta,\nonumber\\
&& \forall l:\ l+q\in L,\ \forall d \in D,\ t \in T \setminus \{T_{\max}\}. \label{eq:mapUB}
\end{eqnarray}

Constraint~\eqref{eq:lambdachoose} selects exactly one slot change value $q\in\Delta$ when a move is executed ($x_{ijt}=1$), and selects none when no move is executed.
Constraint~\eqref{eq:qdef} forces the slot change value to equal the implied value $h_{jt}-h_{it}+m_{ijt}$ as shown in Lemma \ref{lem:slotChange}; when no move is executed (i.e., $x_{ijt}=0$), the big-$M$ terms relax the constraint ($M$ is a sufficiently large number that we can set equal to $2H$ without loss of optimality). Finally, Constraints~\eqref{eq:mapLB}--\eqref{eq:mapUB} connect the destination assignments of railcars removed from track $i$ and railcars placed onto track $j$ through the selected slot change value. Specifically, if $\lambda_{ijqt}=1$ and a railcar with destination $d$ is moved from slot $l$ on track $i$, then a railcar with the same destination $d$ must be placed into slot $l+q$ on track $j$ (i.e., $r_{ij,l+q,dt}=1$). Therefore, these constraints preserve the relative ordering and destination information of railcars during the transfer from track $i$ to track $j$.\\

\noindent 9) Terminal time constraints

\begin{align}
&\sum_{t\in T} u_t = 1, \label{eq:one_u}\\
&w_t = \sum_{\tau=0}^{t} u_\tau && \forall t\in T. \label{eq:w_from_u}
\end{align}

Constraint~\eqref{eq:one_u} guarantees that exactly one time period is identified as the earliest period in which the system reaches a terminal state. Constraint~\eqref{eq:w_from_u} ensures that $w_t=1$  for all periods from the first time the system reaches a terminal state through the end of time horizon.\\

\noindent 10) Terminal condition constraints

\begin{align}
&\sum_{d\in D} \zeta_{kd} - y_{k1t} \le 1-u_t
&& \forall k,\ \forall t\in T, \label{eq:zeta1A}\\
&y_{k1t} - \sum_{d\in D} \zeta_{kd} \le 1-u_t
&& \forall k,\ \forall t\in T, \label{eq:zeta1B}\\
&v_{kl d t} \le \zeta_{kd} + (1-y_{kl t}) + (1-u_t)
&& \forall k,l,d,\ \forall t\in T. \label{eq:zeta2_time}
\end{align}

A terminal state is reached when all railcars with the same destination are consolidated together on a single track, and no further switching operations are required. Constraints~\eqref{eq:zeta1A}--\eqref{eq:zeta1B} force \(\sum_{d\in D}\zeta_{kd}=y_{k1t}\) when $u_t=1$. Therefore, when the system reaches a terminal state at time $t$, if track $k$ is nonempty  (i.e., $y_{k1t}=1$), exactly one destination can be selected for that track; otherwise, if the track is empty, no destination is selected. Constraint~\eqref{eq:zeta2_time} ensures that every occupied slot on track $k$ shares the same destination $d$ when reaching the terminal state. Specifically, if $u_t=1$ and slot $l$ on track $k$ is occupied at time $t$ (i.e., $y_{kl t}=1$), then the constraint reduces to \( v_{kl dt} \le \zeta_{kd}\), which implies that a railcar with destination $d$ cannot occupy slot $l$ on track $k$ unless destination $d$ is selected for track $k$. Therefore, every occupied slot on a nonempty track must have the same destination.\\

\noindent 11) Track destination assignment constraints
\begin{align}
&\sum_{k\in K} \zeta_{kd} - 1 \le (1-u_t)\,|K|
&& \forall d,\ \forall t\in T, \label{eq:onetrackA}\\
&1 - \sum_{k\in K} \zeta_{kd} \le (1-u_t)\,|K|
&& \forall d,\ \forall t\in T. \label{eq:onetrackB}
\end{align}

When the system reaches a terminal state, exactly one outbound train is formed for each destination $d\in D$, which is achieved by \eqref{eq:onetrackA}--\eqref{eq:onetrackB}. Specifically, when $u_t=1$, these contraints force $\sum_{k\in K}\zeta_{kd}=1$, thereby ensuring that each railcar destination is assigned to exactly one track.\\

\noindent 12) Consecutive move constraints

\begin{align}
& \sum\limits_{i\in K}\sum\limits_{\overset{j \in K}{j\ne i}} x_{ij1} \geq \sum\limits_{i\in K}\sum\limits_{\overset{j \in K}{j\ne i}} x_{ij2} \geq \cdots \geq \sum\limits_{i\in K}\sum\limits_{\overset{j \in K}{j\ne i}} x_{ijT}
\label{eq:move_no_stop}
\end{align}

Constraint \eqref{eq:move_no_stop} enforces that 
$x_{ijt}$ is nonincreasing over time, that is, the set of the moves begin in time period 1 and continue without interruption until all railcars have reached their destinations.

\subsection{Time Horizon Determination}
\label{sec:mip_time}
In this subsection, we propose a heuristic method to determine $T_{\max}$ in our RAP MIP model. While $T_{\max}$ must be large enough to guarantee that an optimal solution exists within the time horizon, an excessively large horizon introduces unnecessary variables and computing time. To balance this tradeoff, we construct a polynomial-time heuristic to generate a feasible solution for the RAP. Considering the assumption that at most one switching move may occur per time
period, we then set $T_{\max}$ equal to the total number of moves required by the heuristic to reach a terminal state. Thus, our MIP model determines the minimum number of switching moves within at most $T_{\max}$ time periods.

The proposed heuristic consists of five sequential steps designed to efficiently construct a feasible solution for the RAP. Before describing each step, we present the following definitions. A departure track for destination $d\in D$ is a track designated to contain all railcars with destination $d$ in the terminal state. Let $h_k$ denote the number of railcars on track $k$, for all $k \in K$. A partial track is a track whose number of railcars is greater than 0 but less than the track capacity $H$, i.e., $0 < h_k < H$. A group consists of consecutive railcars with the same destination, and the smallest group size consists of a single railcar. A switch-end group is the group closest to the switch end of the track. Using the illustration in Figure~\ref{slot}, the yellow railcar group is an example of a switch-end group. Recall that a source track is defined as the track from which railcar(s) are moved, while a receiver track is defined as the track to which railcar(s) are moved. 

The main idea of this heuristic is to assign departure tracks for railcars, distinguish departure tracks from non-departure tracks and gradually switch railcars to their designated departure track to reach a terminal state. Under the RAP definition, the terminal state requires that all railcars sharing the same destination are consolidated onto a
single track. However, the flexibility of assigning any track to any destination significantly increases the complexity of the problem. To address this challenge, the heuristic fixes the departure track for destinations $d \in D$. In Step 1, we identify tracks that contain railcars sharing the same destination. For each such track, if the corresponding destination does not yet have a designated departure track, the track is assigned as the departure track for that destination. Thus, the railcars already located on this track can be considered as having already reached their departure track, and therefore do not require additional switching moves, which helps reduce the problem size.

In Step 2, we consolidate railcars by moving them onto non-departure partial tracks in order to increase the number of empty tracks, thereby preparing the yard configuration with more available empty receiver tracks. This facilitates the switching operations performed in Steps 3 and 4. At each iteration of this step, the non-departure partial track with the fewest railcars is selected as the source track, $s$, while the non-departure partial track with the largest number of railcars is selected as the receiver track, $r$. Then, letting $h_s$ and $h_r$ denote the number of cars on tracks $s$ and $r$, respectively, we move $\min(H-h_r, \, h_s)$ railcars from track $s$ to track $r$, such that the capacity constraint of the receiver track is satisfied and the receiver track becomes as full as possible. This process continues until at most one non-departure partial track remains.

Step 3 switches all railcars from the remaining non-departure partial track by iteratively evaluating the switch-end group on this track. If the destination of the switch-end group already has an assigned departure track, the group is directly moved to that departure track. Otherwise, an empty track is selected, the switch-end group is moved to the empty track, and the empty track is assigned as the departure track for that destination.

Step 4 further checks whether there exist any remaining destinations that still do not have designated departure tracks. For each such destination, we randomly selects an empty track and assigns it as the designated departure track. Thus, after this step, each destination $d \in D$ has a designated departure track, which guides the subsequent switching operations in Step 5, i.e., each railcar group can be directly switched to its corresponding destination track.

In Step 5, we iteratively move each switch-end group to the corresponding departure track. The process continues until all non-departure tracks become empty. Since each departure track contains only railcars with its corresponding destination, the terminal state is reached.

Under the assumption that at most one switching move may occur during each time period, the total number of switching moves generated throughout the process is recorded as $Y$, which is then used to estimate $T_{\max}$. The resulting value of $T_{\max}$ provides an upper 
bound for the number of periods required by the MIP model.

\begin{algorithm}[t]
\caption{Heuristic for Determining $T_{\max}$}
\label{alg:Tmax_heuristic}
\begin{algorithmic}[1]

\State Initialize switch move counter $Y \gets 0$, departure track set $\mathcal{P}\gets \emptyset$
\State \textbf{Step 1: Initialize departure tracks}
\For{each track $k\in K$}
    \If{track $k$ is nonempty and all railcars on the track share the same destination $d$}
        \If{destination $d$ has no departure track}
            \State Assign track $k$ as the departure track for destination $d$
            \State Update $\mathcal{P}$
        \EndIf
    \EndIf
\EndFor

\Statex
\State \textbf{Step 2: Consolidate non-departure partial tracks}
\While{there exist at least two non-departure partial tracks}
    \State Select the non-departure partial track with the fewest railcars as source track $s$, where the number of cars on this track is denoted by $h_s$
    \State Select the non-departure partial track with the largest number of railcars as receiver track $r$, where the number of cars on this track is denoted by $h_r$
    \State Move $\min(H-h_r, \, h_s)$ railcars from $s$ to $r$.
    \State $Y \gets Y+1$
\EndWhile

\Statex
\State \textbf{Step 3: Clear railcars on partial tracks}
\While{there exists a non-departure partial track $s$}
    \State Let $d$ be the destination of the switch-end group on $s$
    \If{destination $d$ has the departure track}
        \State Move the switch-end group from $s$ to the departure track of $d$
    \Else
        \State Move the switch-end group from $s$ to an empty track $e$
        \State Assign $e$ as the departure track for destination $d$, update $\mathcal{P}$
    \EndIf
    \State $Y \gets Y+1$
\EndWhile

\Statex
\State \textbf{Step 4: Initialize remaining departure tracks}
\While{there exists a destination $d$ with no departure track}
    \State Randomly assign an empty track as the departure track for destination $d$, update $\mathcal{P}$
\EndWhile

\Statex
\State \textbf{Step 5: Route remaining cars to departure tracks}
\While{there exists a non-departure track containing railcars}
    \State Select a non-departure source track $s$
    \State Let $d$ be the destination of the switch-end group on $s$
    \State Move the switch-end group from $s$ to the departure track of $d$
    \State $Y \gets Y+1$
\EndWhile

\Statex
\State Set $T_{\max}\gets Y$
\State \Return $T_{\max}$
\end{algorithmic}
\end{algorithm}

\section{Zone-DDQN Algorithm}
\label{sec:Zone-DDQN_alg}
In the previous section, we constructed an MIP model for the RAP. However, as the problem size increases, finding an optimal solution within an acceptable computational time becomes challenging. Furthermore, if we restrict each switching move to only one individual railcar and require that the receiving track is either empty or has the same switch-end railcar type as the moved railcar, the the RAP becomes equivalent to the water sorting problem studied by \textcite{ito2023sorting}, which has been proven to be NP-complete. Since the RAP generalizes the water sorting problem by allowing more flexible switching operations, the RAP is NP-hard. Therefore, in this section, we present a Zone-DDQN algorithm that integrates railway-domain heuristics, zone decomposition, and double deep reinforcement learning to efficiently solve the RAP. 

The input of the Zone-DDQN algorithm is the initial yard layout, which includes track information, namely the set of tracks $K$ and track capacity $H$, as well as railcar information, including the set of railcars $R$, the set of available destinations \(D=\{1,2,\dots,|D|\}\), and for each railcar $r \in R$, its destination $d(r) \in D$. Furthermore, the slot position of each railcar is also provided. The output of the Zone-DDQN algorithm is a sequence of switching moves over time. We classify the yard scale into three types: small, medium, and large. The details of yard scale are provided in Section~\ref{sec:c_results}. Zone-DDQN contains three main processes: merging, yard-zone decomposition, and DDQN. We next describe each process in the following subsections, followed by a summary of the overall Zone-DDQN workflow for all problem instances.

\begin{figure}[!htbp]
\centering
\includegraphics[width=1.0\textwidth]{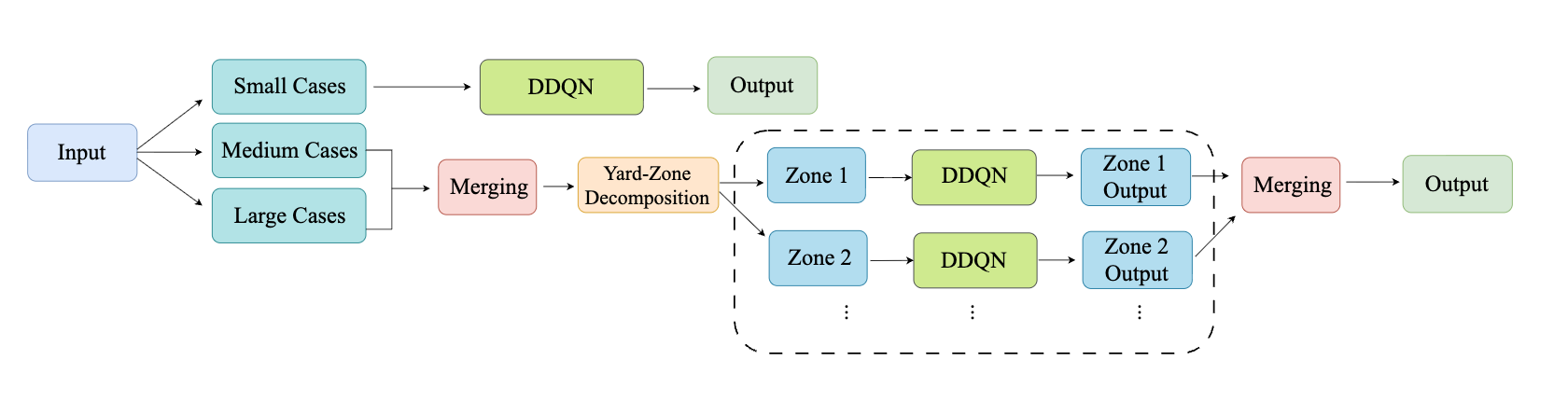}
\caption{Zone-DDQN algorithm}
\label{ZDDQN_algorithm}
\end{figure}

\subsection{Merging}
\label{sec:Merging}
Before describing the merging process, we present the following definitions. For each nonempty track $i \in K$, let \(g_i\) denote its switch-end group, \(l(g_i)\) represent the group length, i.e., the number of railcars contained in \(g_i\). For notation simplicity, we use \(d(g_i)\) to represent the destination of the switch-end group on track $i$. Define a \emph{D-pair} as a pair of distinct nonempty tracks $(i,j)$ such that their switch-end groups share the same destination, i.e., \(d(g_i) = d(g_j)\). The set of all D-pairs is represented by:
\[
\mathcal{D}=\{(i,j) \in K \times K: i \neq j,\ d(g_i) = d(g_j)\}.
\]
A switching move that transfers a switch-end group from one track of a D-pair to the other track in the same D-pair, while satisfying the track capacity constraint, is referred to as a \emph{merging} operation. This operation helps combine switch-end groups sharing the same destination, thereby reducing the overall number of railcar groups in the railyard system. Note that for an initial yard with $|D|$ destinations, the terminal state contains exactly $|D|$ groups, where each group consists of all railcars sharing the same destination.

For each merging operation, we need to determine the merging direction and the set of moved railcars. Specifically, the moved railcars correspond to the switch-end group on the source track. In Algorithm~\ref{alg:group_merging}, a merging direction $i \rightarrow j$ implies that track $i$ is the source track and track $j$ is the receiver track.

When multiple D-pairs are available for merging, we record all feasible merging information in a candidate set $\mathcal{C}$. Each candidate is represented by \[
(\mathrm{src}, \mathrm{rec}, d_{\mathrm{com}}, g_\mathrm{src}, l(g_{\mathrm{src}}), l(g_{\mathrm{rec}}), l(g_{\mathrm{src}})+l(g_{\mathrm{rec}})),
\] where $\mathrm{src}$ and $\mathrm{rec}$ denote the source and receiver tracks, respectively, $d_{\mathrm{com}}$ denotes the common destination label, $g_{\mathrm{src}}$ denotes the moved switch-end group, $l(g_{\mathrm{src}})$ denotes the moved group length, $l(g_{\mathrm{rec}})$ denotes the length of switch-end group on the receiver track, and $l(g_{\mathrm{src}})+l(g_{\mathrm{rec}})$ denotes the combined group size after merging.

For each D-pair, the merging direction is selected by moving the smaller switch-end group to the larger one, shown in Algorithm~\ref{alg:group_merging}, lines 6-9. Using this approach, groups with the same destination gradually become longer through successive merging operations. If the two switch-end groups have the same length, the default merging direction is from the higher-index track to the lower-index track. We then check whether the merging operation in the proposed direction satisfies the receiving track's capacity constraint. If the constraint is satisfied, we add the pair to the candidate set $\mathcal{C}$; otherwise, we evaluate the reverse merging direction, as shown in lines 13-17 of Algorithm~\ref{alg:group_merging}. After all D-pairs are processed, the candidates in $\mathcal{C}$ are sorted in ascending order according to the priority rule defined in Algorithm~\ref{alg:group_merging}. This priority rule is designed to encourage the formation of larger groups $l(g_{\mathrm{src}})+l(g_{\mathrm{rec}})$ as early as possible, which helps consolidate railcars belonging to the same destination into fewer and larger groups across the yard. This rule first prioritizes merging operations that produce the largest combined group size. If multiple candidates have the same combined group size, preference is given to the candidate with larger $l(g_{\mathrm{rec}})$, followed by the larger $l(g_{\mathrm{src}})$. Remaining ties are broken by selecting the candidate with the smaller receiver track index and then the smaller source track index. This rule mainly determines the preferred merging direction. In practice, it can be adjusted according to operational requirements in yards, such as prioritizing specific regions of the yard, balancing workload across tracks, or matching local switching preferences. After sorting, the first candidate is selected and the corresponding merging operation is executed. The procedure terminates when $\mathcal{C}$ becomes empty.

\begin{algorithm}[H]
\caption{Merging Operations}
\label{alg:group_merging}
\begin{algorithmic}[1]

\While{true}
    \State Initialize candidate set $\mathcal{C} \leftarrow \emptyset$

    \For{each D-pair $(i,j)$ with $i<j$}
        \State Let $g_i$ and $g_j$ denote the switch-end groups on tracks $i$ and $j$
        \State Let $l(g_i)$ and $l(g_j)$ denote the corresponding group lengths, respectively
        
        \If{$l(g_i) < l(g_j)$}
            \State Merging direction $i \rightarrow j$
        \ElsIf{$l(g_j) < l(g_i)$}
            \State Merging direction $j \rightarrow i$
        \Else
            \State Merging direction from higher-index track to lower-index track
        \EndIf

        \If{executing the merging operation in the proposed merging direction satisfies receiver track's capacity}
            \State Add merge candidate into $\mathcal{C}$
        \ElsIf{reverse direction satisfies track capacity}
            \State Add reverse merge candidate into $\mathcal{C}$
        \EndIf
    \EndFor

    \If{$\mathcal{C} = \emptyset$}
        \State \textbf{break}
    \EndIf

    \State Sort candidates in $\mathcal{C}$ in ascending order according to
    \[
    \left(
    -(l_{\mathrm{src}}+l_{\mathrm{rec}}),
    -l_{\mathrm{rec}},
    -l_{\mathrm{src}},
    \mathrm{rec},
    \mathrm{src}
    \right)
    \]

    \State Select and execute the first candidate in $\mathcal{C}$

\EndWhile

\end{algorithmic}
\end{algorithm}

\newpage
\subsection{Yard-Zone Decomposition}
\label{sec:zone_decompose}
After the merging procedure, the yard may contain several long switch-end groups. To reduce computational complexity, we decompose the large-scale yard into smaller sub-yards that can be processed independently. Each sub-yard is referred to as a zone, and the corresponding RAP within a zone is referred to as a RAP-subproblem. For a RAP-subproblem, a sub-terminal state is defined as the state in which all railcars within the zone sharing the same destination are consolidated to a single track.

Given a yard consisting of $|K|$ tracks and $|D|$ destination labels for the RAP, we decompose the yard into several consecutive parallel zones while keeping all railcars in their original positions. Each zone consists of a subset of consecutive tracks and is solved independently as a RAP-subproblem. By partitioning the yard into smaller zones, the RL agent described in Section~\ref{sec:DDQN_PHASE} for medium and large scale yards only needs to learn switching operations within each zone rather than across the entire yard, which significantly reduces the state and action space.

For a RAP-subproblem, if the number of destination labels within a zone exceeds the number of tracks in that zone, the corresponding sub-terminal state becomes infeasible, since at least one destination would be unable to obtain an exclusive track. 

\begin{proposition}
\label{prop:rap_sub_feasible}
If each zone contains at least $|D|$ tracks, then the corresponding RAP-subproblem is feasible.
\end{proposition}

\begin{proof}
Each zone contains at least $|D|$ tracks, while the number of destinations in a zone is at most $|D|$. Therefore, one distinct track can be assigned to each destination within the zone. Hence, the corresponding RAP-subproblem is feasible.
\end{proof}

To ensure that RAP-subproblems remains feasible in each zone, the yard decomposition strategy, including the number of zones and the number of tracks assigned to each zone, is critical. 

Let $\mathcal{F}$ denote the number of zones. Its value is determined by
\[
\mathcal{F} = \left\lfloor \frac{|K|}{|D|} \right\rfloor,
\]
where $\lfloor \cdot \rfloor$ denotes the floor function. Thus, $\mathcal{F}$ equal to the largest integer less than or equal to $\frac{|K|}{|D|}$. Considering the assumption (v) illustrated in Section~\ref{sec:mip}, the value of $\mathcal{F}$ should be greater than or equal to 1. The Yard-Zone decomposition algorithm is illustrated in Algorithm~\ref{alg:zone_decomposition}.

\begin{algorithm}[H]
\caption{Yard-Zone Decomposition}
\label{alg:zone_decomposition}
\begin{algorithmic}[1]

\State Compute
\[
\mathcal{F} \gets \left\lfloor \frac{|K|}{|D|} \right\rfloor
\]

\If{$\mathcal{F}=1$}
    \State Return one zone containing all tracks in $K$
\EndIf

\If{$\mathcal{F}>1$ and $\mathcal{F}=\frac{|K|}{|D|}$}
    \State Divide the yard into $\mathcal{F}$ consecutive zones
    \State Assign exactly $|D|$ consecutive tracks to each zone
\EndIf

\If{$\mathcal{F}>1$ and $\mathcal{F}<\frac{|K|}{|D|}$}
    \State Divide the yard into $\mathcal{F}$ consecutive zones
    \State Assign exactly $|D|$ consecutive tracks to each of the first $\mathcal{F}-1$ zones
    \State Assign all remaining tracks to the last zone
\EndIf

\end{algorithmic}
\end{algorithm}

When $\mathcal{F}=1$, all tracks in $K$ are treated as a single zone. Recall Assumption (v) of the RAP, which states that \(|D| \leq |K|\). Therefore, according to Proposition~\ref{prop:rap_sub_feasible}, the corresponding RAP-sub is feasible.

When $\mathcal{F}>1$ and the number of tracks in the yard can be evenly divided by the number of destinations, i.e., \(\mathcal{F}=\frac{|K|}{|D|}\), the yard is partitioned into $\mathcal{F}$ consecutive zones, where each zone contains exactly $|D|$ consecutive tracks. Specifically, given \(K=\{0,1,2,\dots,|K|-1\}\), tracks \(\{0,1,2,\dots,|D|-1\}\) are assigned to Zone 1. The next $|D|$ consecutive tracks \(\{|D|,|D|+1,|D|+2,\dots,2|D|-1\}\) are assigned to Zone 2, and so on. Therefore, all zones have identical sizes. When $\mathcal{F}>1$ but the yard cannot be evenly divided by the number of destinations, i.e.,
\(\mathcal{F}<\frac{|K|}{|D|}\), the first $\mathcal{F}-1$ zones are assigned exactly $|D|$ consecutive tracks, while the remaining tracks are assigned to the final zone. Consequently, the final zone may contain more than $|D|$ tracks. Under both cases above, every zone contains at least $|D|$ tracks. Therefore, according to Proposition~\ref{prop:rap_sub_feasible}, the RAP-subproblem in each zone is feasible.

\newpage
\subsection{DDQN}
\label{sec:DDQN_PHASE}
In this section, the basic principles of the DDQN for the RAP are presented. Before describing the DDQN formulation for the RAP, we first briefly introduce the background of DDQN. 

The popular Q-learning \parencite{watkins1992q} methods have been designed to solve sequential decision-making problems. The general framework of Q-learning is based on an MDP involving a dynamic sequence of decisions.  An MDP typically contains a number of possible states, and with each state is associated a potential actions.  Selecting an action in a state produces a reward and results in a transition to another state.  The quality of an action in a given state is given by the action value function \(Q(s,a)\). Actions are typically guided by a decision policy $\pi$, which determines an action $a_t$ when in state $s_t$, i.e.,
\[
Q_{\pi}(s,a) \;=\; \ \mathbb{E} \!\left[\sum_{t\ge 0}\gamma^t\, r(s_t,a_t,s_{t+1})
\ \Bigm|\ s_0=s,\ a_0=a\right],
\]
where \(r(s_t,a_t,s_{t+1})\) represents the immediate reward associated with transferring from state \(s_t\) to \(s_{t+1}\), and \(\gamma \in [0, 1)\) denotes a discounting factor. The optimal value of \(Q(s,a)\), denoted as \(Q^*(s,a)\), corresponds to the maximum value of \(Q(s,a)\) over all policies, i.e., \(Q^*(s,a)=\max_{\pi} Q_{\pi}(s,a)\). Classical Q-learning methods use lookup tables to store and update these Q-values for every state-action pair, which is not suitable for large-scale problems. Instead of storing all possible Q-values in a table, Deep Q-Networks (DQN) use deep neural networks to approximate Q-values \parencite{mnih2015human}.  These neural networks use a trainable weight and bias vector $\theta$ to estimate Q-values using the function $Q(s,a;\theta)$.

\textcite{mnih2015human} also use the two important concepts of experience replay and the target network. Experience replay stores previous transitions \((s_t,a_t,r_t,s_{t+1})\) in a replay buffer. Random samples are taken from the replay buffer during the training phase to update parameter values and reduce the correlation among successive states.  The target network is a separate neural network from the main network, with parameters $\theta^{-}$ that are only updated periodically based on the values obtained in the main network to promote convergence. Otherwise, DQNs tend to overestimate Q-values when the same network is used to select and evaluate the quality of actions. To address this issue, \textcite{van2016deep} proposed DDQN, which separates the choice of action and the corresponding evaluation of the quality of the action. Specifically, the idea is to use the main network for choosing an action in a state and the target network evaluates the chosen action, thereby reducing overestimation bias and improving training stability. Motivated by these advantages, we employ DDQN to solve the RAP. We start by constructing the MDP formulation, and then illustrate the neural network architecture for the RAP. 

\subsubsection{MDP}
\label{sec:MDP_RAP}
In RAP settings, the locomotive can be considered as the RL agent responsible for selecting switching actions, while the environment corresponds to the entire railyard system, including the track and railcar information. We model the RAP as an MDP. At every decision point, the agent observes the current yard state, selects a feasible switching action, and transitions to a new state according to the environment dynamics, while receiving an immediate reward. Recall that \(K=\{0,1,2,\dots,|K|-1\}\). The details are as follows.\\

\noindent \paragraph{(1) State.} A state \(s_{t}\in\mathcal{S}\), where \(\mathcal{S}\) is the state space, represents railcars and track information at time $t$, which is represented by a tuple
\[
  s_t=\bigl(c_0, c_1,c_2,\dots,c_{|K|-1}\bigr),
\]
where for each track \(i\in K\), \(c_i\) contains an ordered list of railcars on
track \(i\) from the \emph{switch end} to the \emph{dead end} without gaps. Initial state \(s_0\) represents the given initial yard layout. A \emph{terminal} state is any configuration in which all railcars sharing the same destination are consolidated onto a
single track.\\

\noindent \paragraph{(2) Action.} An action \(a_{t}\in\mathcal{A}(s_{t})\), where \(\mathcal{A}\) denotes the action space, corresponds to a switching move characterized by
\[
a_t = (i,j,m),
\]
where \(i\in K\) is the source track, \(j\in K\setminus\{i\}\) is the receiver track, and
\(m\) represents the number of the moved railcars from $i$ to $j$. The ordering of the moved railcars remain unchanged.

To reduce the action space and avoid infeasible or ineffective switching operations, several action masks are imposed:

\begin{enumerate}
    \item An action that splits a group is not allowed.

    \item The receiver track must satisfy the track capacity constraint after the move.

    \item A track that already contains all railcars with a given destination is protected and cannot be selected as a source track.

    \item Moving all railcars from a source track to an empty receiver track is prohibited, since such operations only relocate the entire track configuration without improving the yard state.

    \item Immediate reverse moves are prohibited. Specifically, if the previous action moves \(m\) railcars from track \(i\) to track \(j\), then the reverse action moving the same \(m\) railcars from track \(j\) back to track \(i\) is not allowed in the next step.
\end{enumerate}

\noindent \paragraph{(3) State Transition.} Given a state \(s_{t}\), taking an action \(a_t\) deterministically transitions the system to the next state \(s_{t+1}\). The moved railcars are simply push into the receiver track. Consider the following state:
\[
s_t=\bigl([1,1,2,2],\ [2,3],\ []\bigr),
\]
where Track 0 contains railcars \([1,1,2,2]\), Track 1 contains railcars \([2,3]\), and Track 2 is empty. Suppose action
\(a_t=(0,1,2)\) is selected, which moves the leftmost two railcars \([1,1]\) from Track 0 to Track 1. The next state becomes
\[
s_{t+1}=\bigl([2,2],\ [1,1,2,3],\ []\bigr).
\]

\noindent \paragraph{(4) Reward Function.} Recall that objective of the RAP is to minimize the total number of switching moves required to reach a terminal state. To achieve this objective while encouraging the formation of a terminal state gradually throughout the switching process, we define a mixed reward function given by

\begin{equation}
r(s_t,a,s_{t+1})
=
-1
+
w_{\mathrm{deg}}r_{\mathrm{deg}}
+
w_{\mathrm{mono}}r_{\mathrm{mono}}
+
\begin{cases}
Bonus, & \text{if } s' \text{ is terminal},\\
0, & \text{o.w.},
\end{cases}
\label{eq:rap_reward}
\end{equation}
where the first term ($-1$) denotes the penalty associated with each switching move, and the last term \(B>0\) is a terminal state bonus.

The terms \(r_{\mathrm{deg}} \) and \(r_{\mathrm{mono}} \) correspond to binary variables, which represents a signal that captures a state change from $s_t$ to $s_{t+1}$ after taking an action $a$. Specifically, \(r_{\mathrm{deg}}\) indicates whether the total number of railcar groups decreases from state \(s_t\) to state \(s_{t+1}\). If the total number of groups decreases after executing action \(a\), then \(r_{\mathrm{deg}} =1\); otherwise, it is 0. Similarly, \(r_{\mathrm{mono}} =1\) if the total number of tracks containing all railcars of a single destination increases; otherwise, \(r_{\mathrm{mono}} = 0\). The terms \(w_{\mathrm{deg}}\) and \(w_{\mathrm{mono}}\) correspond to weight parameters.

\subsubsection{Neural Network Architecture}
\label{sec:neral_network}
The current yard state is used as input to the neural network, which predicts the Q-values corresponding to all feasible switching moves as the output. Based on the number of tracks \(|K|\) and track capacity \(H\), we construct a \(|K|\times H\) matrix \(V(s)\) to represent the yard layout, where each entry $v_{ij}(s)$ denotes the destination label of the railcar located in row \(i\) and column \(j\).
Empty slots are represented by a 0.

Railcars destinations are right-aligned so that the dead end always corresponds to the last column (\(h=H\)). Therefore, for a track containing \(h_k\) railcars, the switch-end railcar is located in column \(H-h_k+1\), while the dead-end railcar is located in column \(H\).
Considering the above example shown in the state transition with a track capacity of 4, we have 
\[
V(s)=
\begin{bmatrix}
1 & 1 & 2 & 2 \\
0 & 0 & 2 & 3 \\
0 & 0 & 0 & 0 
\end{bmatrix}.
\]

\noindent Each entry \(v_{ij}(s)\) is mapped to a \(p\)-dimensional embedding vector \(e(v_{ij}(s))\in \mathbb{R}^{p}\). Here, $p$ is the
embedding dimension, a tunable parameter that controls the size of the vector used to represent each entry of matrix \(V(s)\).  For each track \(i\in K\), let
\begin{equation}
z_i \in \mathbb{R}^{H\times p}
\label{eq:track_repres}
\end{equation}
denote the track-level embedding vector obtained by concatenating the embedding vectors of all \(H\) positions on track \(i\). 

The proposed neural network architecture consists of an encoder and a decoder. The encoder extracts latent structural representations from the embedded yard state, while the decoder maps the learned latent representations into Q-values for all enumerated actions.

The encoder first transforms \(z_i\) into an initial embedding representation \(h_i^{(0)}\) through a track-level multilayer perceptron (MLP):
\begin{equation}
h_i^{(0)} = \mathrm{MLP}_{\mathrm{track}}(z_i).
\label{eq:track_MLP_encoder}
\end{equation}
Then, all track embeddings are concatenated into a global yard representation:
\begin{equation}
h^{(0)}
=
[h_0^{(0)},h_1^{(0)},\dots,h_{|K|-1}^{(0)}].
\label{eq:STATE_MLP_encoder}
\end{equation}
The global yard representation is then processed by a state-level MLP, which provides estimated Q-values for enumerated actions:
\begin{equation}
Q(s,a;\theta)
=
\mathrm{MLP}_{\mathrm{state}}(h^{(0)}).
\label{eq:q_VALUE_DQN}
\end{equation}

During training, traditional \(\epsilon\)-greedy exploration is used to select feasible actions, i.e., with probability \(\epsilon\), we randomly choose an action from the feasible action set \(\mathcal{A}(s_t)\); otherwise, we choose the action with the largest current estimate of Q-value.  In particular, 

\begin{equation}
a_t=
\begin{cases}
\text{Uniform}(\mathcal{A}(s_t)), & \text{with probability } \epsilon,\\[2mm]
\arg\max\limits_{a\in \mathcal{A}(s_t)} Q(s_t,a;\theta), & \text{with probability } 1-\epsilon.
\end{cases}
\label{eq:epsilon_greedy}
\end{equation}

To stabilize the training process, two neural networks are maintained: the online (main) network parameterized by \(\theta\) and the target network parameterized by \(\theta^{-}\). The online network selects actions and updates  Q-functions, while the target network computes the Q-values. After every $N_{target}$ training steps, the target network parameters are copied using the parameters from the online network.

Following the DDQN framework \parencite{van2016deep}, the target Q-value for RAP is computed as
\begin{equation}
y_t
=
r(s_t,a_t,s_{t+1})
+
\gamma
(1-\mathrm{done}_{t+1})
Q\!\left(
s_{t+1},
\arg\max\limits_{a'\in\mathcal{A}(s_{t+1})}
Q(s_{t+1},a';\theta);
\theta^{-}
\right),
\label{eq:ddqn_target}
\end{equation}
where \(y_t\) denotes the DDQN target Q-value used for training at time step \(t\), and \(\mathrm{done}_{t+1}\) is a binary indicator variable that equals 1 if state \(s_{t+1}\) is terminal and 0 otherwise.

The loss function is calculated by 
\begin{equation}
L(\theta)
=
\mathbb{E}
\left[
\left(
y_t
-
Q(s_t,a_t;\theta)
\right)^2
\right].
\label{eq:ddqn_loss}
\end{equation}
The online network parameters are then updated by minimizing the loss function
\begin{equation}
\theta
\leftarrow
\theta
-
\alpha
\nabla_{\theta}L(\theta)
\label{eq:ddqn_theta_update}
\end{equation}
where the learning rate, $\alpha \in (0,1]$, determines the magnitude of the parameter updates.

The training procedure is presented in Algorithm~\ref{alg:DDQN_RAP}. During each episode, the DDQN chooses feasible switching movements according to Equation~\eqref{eq:epsilon_greedy}.  We maintain the corresponding transitions in the replay buffer and randomly sample mini-batches to train the online network. The online network is used to identify the next action, while the corresponding Q-value is estimated in the target network. As noted earlier, we periodically update the target network parameter values based on the online network parameters to improve training stability and convergence.  Training  terminates when either a terminal state is reached or the maximum number of steps \(t_{\max}\) is exceeded. Note that we assume at most one switching move can be executed during each time step. 

Upon training completion, switching decisions for the RAP are determined by the trained online network. Starting from $s_0$, at every step, we simply choose the switching move with the highest Q-value among all feasible actions:
\begin{equation}
a_t \in \arg\max_{a\in \mathcal{A}(s_t)} Q(s_t,a;\theta),
\label{eq:ddqn_greedy_evaluation}
\end{equation}
This procedure continues until the system enters a terminal configuration.

\begin{algorithm}[H]
\caption{DDQN for RAP}
\label{alg:DDQN_RAP}
\begin{algorithmic}[1]

\State Initialize replay buffer $\mathcal{B}$
\State Initialize $Q(s,a;\theta)$, $Q(s,a;\theta^-)$ with $\theta^-=\theta$
\State Initialize exploration rate $\epsilon$

\For{episode $=1,\dots,N_{\mathrm{episode}}$}

    \State Reset environment and obtain $s_0$

    \For{$t=0,\dots,t_{\max}$}

        \State Generate feasible action set $\mathcal{A}(s_t)$

        \If{$\mathcal{A}(s_t)=\emptyset$}
            \State break
        \EndIf

        \State Select action $a_t$ using Eq.~(\ref{eq:epsilon_greedy})

        \State Perform selected $a_t$

        \State Observe reward $r_t$, $s_{t+1}$, and terminal indicator $\mathrm{done}_{t+1}$

        \State Generate feasible action mask for $s_{t+1}$

        \State Store transition
        \[
        (s_t,a_t,r_t,s_{t+1},\mathrm{done}_{t+1})
        \]
        into replay buffer $\mathcal{B}$

        \If{replay buffer size exceeds threshold and \(\mathrm{global\_step} \bmod N_{\mathrm{train}} = 0\)}

            \State Randomly sample a mini-batch from $\mathcal{B}$

            \State Calculate DDQN target value using Eq.~(\ref{eq:ddqn_target})

            \State Update online network parameters $\theta$ by minimizing Eq.~(\ref{eq:ddqn_loss})

        \EndIf

        \If{\(\mathrm{global\_step} \bmod N_{\mathrm{target}} = 0\)}
            \State Update target network:
            \[
            \theta^- \leftarrow \theta
            \]

        \EndIf

        \If{$\mathrm{done}_{t+1}=1$}

            \State break

        \EndIf

    \EndFor

    \State Update exploration rate $\epsilon$:
  \(
    \varepsilon \gets \max\left(\varepsilon_{\min},\, \varepsilon \cdot \varepsilon_{\mathrm{decay}}\right).
  \)

\EndFor

\end{algorithmic}
\end{algorithm}

\newpage
\subsection{Zone-DDQN Workflow}
\label{sec:zone_DDQN_workflow}
Figure~\ref{rapyard} illustrates the overall workflow of the proposed Zone-DDQN framework for the RAP. Depending on the problem size, different solution strategies are adopted.

For small-scale RAP instances, the DDQN framework is directly applied to the entire yard system to generate switching decisions. 

For medium- and large-scale RAP instances, directly applying DDQN to the entire yard becomes computationally challenging because the number of possible states and actions increases rapidly. To improve scalability, the preprocessing merging heuristic described in Section~\ref{sec:Merging} is first applied to repeatedly merge the switch-end groups sharing the same destination. The resulting yard state is then decomposed into multiple yard zones according to the proposed yard-zone decomposition strategy presented in Section~\ref{sec:zone_decompose}.

Each zone is independently solved using the DDQN framework illustrated in Section~\ref{sec:DDQN_PHASE}. Specifically, DDQN training is performed within each zone to learn switching policies for the corresponding subproblem. After training, the learned online network is used to greedily generate switching actions for the zone until a sub-terminal state is reached, which means each track within a zone contains at most one single destination label.  

After all zones are sequentially processed, a final merging heuristic is applied. This additional merging step is necessary because railcars with the same destination may still appear across different zones. Thus, we further apply merging operations to ensure all railcars sharing the same destination are consolidated onto a single track. 

The entire process transforms the initial yard state into a terminal yard state. The overall number of switching moves generated during the Zone-DDQN framework is used as the objective function value of the RAP solution.

\section{Computational Results}
\label{sec:c_results}
Computational results are presented in this section. The proposed MIP model and Zone-DDQN algorithm are implemented and compared in terms of solution quality and computational performance. All experiments are conducted on a laptop equipped with an Apple M3 Pro chip and 18 GB of memory. The MIP model is implemented in Python 3.12.11 and solved using Gurobi Optimizer version 12.0.3.

Considering the number of tracks, track capacity, and the number of railcar destinations, we define three yard-size scales for computational experiments, as shown in Table~\ref{tab:scale_tracks_capacity}, following the general flat yard-scale description in \textcite{wong1981railroad}.

\begin{table}[ht]
\centering
\caption{Problem scales by yard size}
\label{tab:scale_tracks_capacity}
\begin{tabular}{lccc}
\toprule
 & \textbf{Small} & \textbf{Medium} & \textbf{Large} \\
\midrule
\textbf{Number of tracks ($|K|$)} & 5 &  15 & 30 \\
\textbf{Capacity of each track ($H$)} &30  & 40 & 60 \\
\textbf{Number of destinations ($|D|$)} &3  & 6 & 9 \\
\bottomrule
\end{tabular}
\end{table}

An overall set of 30 problem instances are generated, including 10 instances for every yard-size scale. 
Within each generated instance, the number of railcars associated with destination \(d\in D\) is randomly generated according to \(n_d \sim \mathrm{Uniform}(1,H)\). Then, all railcars are randomly distributed across tracks.

Table~\ref{tab:ddqn_parameters_scale} presents the DDQN parameters used at different scales. The medium-scale and large-scale instances share the same parameter settings. This is because both instance categories have relatively large state and action spaces compared with the small-scale instances. Across all problem scales, we set the weights \(w_{\mathrm{deg}} = 4\) and \(w_{\mathrm{mono}}=3\) in the MDP reward function, the discount factor \(\gamma=0.99\), the initial exploration rate \(\epsilon=1.0\), and the embedding dimension \(p=32\).

\begin{table}[ht]
\centering
\caption{Scale-dependent DDQN parameter settings}
\label{tab:ddqn_parameters_scale}
\begin{tabular}{lcc}
\toprule
\textbf{Parameter} & \textbf{Small} & \textbf{Medium and Large} \\
\midrule
Number of training episodes (\(N_{\mathrm{episode}}\)) & 1000 & 2000  \\
Bonus for terminal & 15 & 100  \\
Maximum steps per episode (\(t_{\max}\)) & 40 & 200  \\
Replay buffer capacity & 100,000 & 150,000 \\
Mini-batch size & 128 & 64  \\
Learning rate (\(\alpha\))& \(1\times10^{-3}\) & \(1\times10^{-4}\)  \\
Target network update frequency (\(N_{\mathrm{target}}\)) & 500 & 1000  \\
Minimum exploration rate (\(\epsilon_{\min}\)) & 0.02 & 0.05 \\
Exploration decay rate (\(\epsilon_{\mathrm{decay}}\)) & 0.995 & 0.998 \\
\bottomrule
\end{tabular}
\end{table}

To compare the performance of the MIP model and the proposed Zone-DDQN algorithm, the following evaluation metrics are used. The objective value (Obj.) represents the overall number of switching movements needed to reach a terminal yard state. Computational time (Time) is reported in seconds. The time ratio for the MIP model is calculated as \(\text{Time Ratio}= \frac{\text{MIP Time}}{\text{Zone-DDQN Time}}\). In addition, the optimality gap is computed as
\(\text{Optimality Gap (\%)} = \frac{\text{Zone-DDQN Obj.}-\text{MIP Obj.}}{\text{MIP Obj.}}
\times 100\%\).

\begin{table}[htbp]
    \centering 
    \footnotesize
    \caption{MIP and Zone-DDQN algorithm results}
    \label{tab:3rd_results}
    \begin{tabular}{lcccccc}
        \toprule
        & \multicolumn{3}{c}{MIP} & \multicolumn{3}{c}{Zone-DDQN} \\
        \cmidrule(lr){2-4} \cmidrule(lr){5-7}
        Instances 
        & Obj. & Time (s) & Time Ratio & Obj. &  Time (s) & Optimality Gap (\%) \\
        \midrule
        
        Small (7 cases) & 4.42 & 6693.91 & 51.36 & 4.71 & 93.47 & 5.71 \\
        
        Small (3 cases) & - & - & - & 16 & 345.87 & - \\
        
        Medium (10 cases) & - & - & - & 67 & 177.19 & - \\
        
        Large (10 cases) & - & - & - & 80.2 & 214.42 & - \\
        
        \midrule
        \textbf{Overall Average} & 4.42 & 6693.91 & 51.36 & 41.98 & 207.74 & 5.71 \\
        \bottomrule
        \multicolumn{7}{l}{\footnotesize \textit{Note:} “-” indicates MIP cannot obtain results within the 24-hour time limit.}
        
    \end{tabular}
\end{table}

By setting a running time limit of 24 hours, 7 instances of small instances are solvable. The remaining 3 small-scale instances, as well as all medium- and large-scale instances, cannot be solved by the MIP model. However, Zone-DDQN can generate solutions for all instances. The average running time for MIP across 7 cases is significantly large, which is 6693.91 seconds as shown in Table~\ref{tab:3rd_results}. Instead, Zone-DDQN only requires \(93.47\) seconds. 

By setting a running time limit of 24 hours, 7 small-scale instances are solvable by the MIP model. The remaining 3 small-scale instances, as well as all medium- and large-scale instances, cannot be solved by the MIP model within the prescribed time limit. However, the Zone-DDQN heuristic can generate solutions for all instances. The average running time of the the MIP model across the 7 solvable cases is significantly large, reaching \(6693.91\) seconds as shown in Table~\ref{tab:3rd_results}, which is \(51.36\) times larger than that of Zone-DDQN. In contrast, Zone-DDQN only requires an average of \(93.47\) seconds while maintaining an average optimality gap of \(5.714\%\). The relatively low optimality gap demonstrates the solution quality of Zone-DDQN. The overall average running time across all instances for Zone-DDQN is 207.74 seconds, which indicates the efficiency and scalability of Zone-DDQN. Therefore, as the problem scale increases, the proposed Zone-DDQN algorithm exhibits significantly better scalability and computational performance than the MIP model.

\section{Conclusion and Future Research}
\label{sec:conclusion}
In this paper, we study the Railcar Assignment Problem (RAP) in flat switching yards and formulate the problem as a sequential railcar switching optimization problem. An MIP model is first developed to minimize the total number of switching movements required to transform an initial yard state into a terminal yard state. We further establish the NP-hardness of the RAP by reducing the water sorting problem to a restricted version of RAP. Given this complexity, a Zone-DDQN framework is developed by decomposing the yard into multiple smaller yard zones and solving each zone independently using DDQN. The proposed framework integrates preprocessing merging heuristics, yard-zone decomposition, DDQN-based switching policy learning, and final merging operations. The preprocessing merging heuristic helps reduce the total number of groups in the railyard by merging switch-end groups that share the same destination. After yard decomposition, DDQN is independently implemented within each zone to generate a sub-terminal state. Finally, the merging heuristic is applied again to further combine railcars with the same destination across different zones into a single track.

Computational experiments demonstrate the effectiveness and scalability of the proposed Zone-DDQN framework across 30 generated RAP instances ranging from small- to large-scale yards. The overall average running time of Zone-DDQN is \(207.74\) seconds. Although the MIP model is able to solve several small-scale instances, the average running time is significantly larger, reaching \(6693.91\) seconds. Furthermore, the average optimality gap of Zone-DDQN is only \(5.71\%\), indicating that the proposed framework can achieve competitive solution quality with substantially better computational efficiency.

Future research may extend the RAP to more complex train configurations where railcars can be stacked in two vertical layers, referred to as double-stack trains.  In such settings, additional operational constraints need to be considered to ensure safety and stability. For example, heavier railcars are generally required to be placed in the lower layer, while lighter railcars are placed in the upper layer. Another direction is the consideration of switching operations for double-stack trains in different types of rail yards. This paper focuses on stub flat yards, where switching operations are performed from only one side of each track. Future studies may consider through yards, where switching operations can be conducted from both ends of the track, resulting in different operational characteristics and switching strategies.

\printbibliography

\newpage
\noindent {\Large \textbf{Appendix}}
\begin{appendices}
\section{Proofs}\label{sec:proofs_appen}

\begin{proof}[\unskip\nopunct] \noindent \textbf{Proof of Lemma \ref{lem:slotChange}.} 
Consider a railcar that is among the $m_{ijt}$ railcars moved from track $i$ to track $j$ in period $t$. Since exactly $m_{ijt}$ railcars leave track $i$, these railcars must be the $m_{ijt}$ railcars closest to the switch end of track $i$ due to the one-sided access property of the track. Let $p_{rt}$ indicate the slot position for railcar $r$ in period $t$. In particular, these removed railcars should have a slot index greater than \( h_{it}-m_{ijt}\), i.e., \( p_{r,{t-1}} > h_{it}-m_{ijt}\).

After the move, the railcar is placed on track $j$. Since track $j$ contains $h_{jt}$ railcars before receiving the moved railcars, the new slot position of the railcar on track $j$ is \( p_{rt} = h_{jt} + p_{t-1} - \left(h_{it}-m_{ijt}\right) \).

Thus, the slot change is
\(p_{rt}-p_{r,{t-1}}
=
h_{jt}-h_{it}+m_{ijt}.\)

Therefore, the slot change for each moved railcar is
\[
q=h_{jt}-h_{it}+m_{ijt}.
\]

Because each track has capacity $H$, a feasible move satisfies
\[
1\leq m_{ijt}\leq h_{it}, 
\qquad 
h_{jt}+m_{ijt}\leq H.
\]
Thus, the smallest possible value of $q$ is obtained when $h_{jt}=0$, $h_{it}=H$, and $m_{ijt}=1$, which gives
\[
q_{\min}=-(H-1).
\]
The largest possible value is obtained when $h_{jt}=H-1$, $h_{it}=1$, and $m_{ijt}=1$, which gives
\[
q_{\max}=H-1.
\]
Therefore,
\[
q\in\{-(H-1),\ldots,H-1\}=\Delta.
\]
This completes the proof.
\end{proof}

\end{appendices}
\end{document}